\documentclass[a4,12pt]{amsart}
\usepackage{amssymb,amstext,amsmath,amscd,amsthm,amsfonts,enumerate,latexsym}
\usepackage{color}
\usepackage[dvipdfmx]{graphicx}
\usepackage[all]{xy}
\theoremstyle{plain}
\newtheorem{thm}{Theorem}[section]

\newtheorem*{thm*}{Theorem}
\newtheorem*{cor*}{Corollary}

\newtheorem{prop}[thm]{Proposition}

\newtheorem{lem}[thm]{Lemma}
\newtheorem{cor}[thm]{Corollary}

\newtheorem*{claim*}{Claim}

\theoremstyle{definition}

\newtheorem{ex}[thm]{Example}

\newtheorem{rem}[thm]{Remark}

\newtheorem*{acknowledgments}{Acknowledgments}

\theoremstyle{remark}

\newtheorem*{proof of claim}{{\sl Proof of Claim}}

\numberwithin{equation}{thm}

\def\rank{\mathrm{rank}}

\newcommand{\rme}{\mathrm{e}}

\newcommand{\calM}{\mathcal{M}}

\newcommand{\calP}{\mathcal{P}}

\newcommand{\calR}{\mathcal{R}}
\newcommand{\calS}{\mathcal{S}}
\newcommand{\calT}{\mathcal{T}}

\newcommand{\fkm}{\mathfrak{m}}

\newcommand{\fkp}{\mathfrak{p}}

\newcommand{\mapright}[1]{%
\smash{\mathop{%
\hbox to 1cm{\rightarrowfill}}\limits^{#1}}}

\newcommand{\mapleft}[1]{%
\smash{\mathop{%
\hbox to 1cm{\leftarrowfill}}\limits_{#1}}}

\def\depth{\operatorname{depth}}

\def\Supp{\operatorname{Supp}}

\def\Ass{\operatorname{Ass}}
\def\height{\mathrm{ht}}

\def\Spec{\operatorname{Spec}}

\def\ol{\overline}

\def\Card{\operatorname{Card}}

\title{ideals of reduction number two}
\author{Shinya Kumashiro}
\address{Department of Mathematics and Informatics, Graduate School of Science and Engineering, Chiba University, Yayoi-cho 1-33, Inage-ku, Chiba, 263-8522, Japan}

\thanks{The author was supported by JSPS KAKENHI Grant Number JP19J10579 and JSPS Overseas Challenge Program for Young Researchers.}

\begin{document}

\begin{abstract}
In a local Cohen-Macaulay ring $(A, \mathrm{m})$, we study the Hilbert function of an $\mathrm{m}$-primary ideal $I$ whose reduction number is two. It is a continuous work of the papers of Huneke, Ooishi, Sally, and Goto-Nishida-Ozeki. With some conditions, we show the inequality $\mathrm{e}_1(I)\ge \mathrm{e}_0(I) - \ell_A (A/I) + \mathrm{e}_2(I)$ of the Hilbert coefficients, which is the converse inequality of Sally and Itoh. We also study relations between the Hilbert coefficients and the depth of the associated graded ring. 
\end{abstract}

\maketitle



\section{Introduction}\label{section1}

The purpose of this paper is to study the Hilbert function.
It is well-known that the Hilbert function is deeply related to the structure of the associated graded ring, the Rees algebra, and the given ring. 
One can consult \cite{RV} for the study of Hilbert functions.

Let $(A, \fkm)$ be a local Cohen-Macaulay ring of dimension $d>0$ and $I$ an $\fkm$-primary ideal. Assume that the residue field is infinite, and let $Q$ be a parameter ideal of $A$ such that $Q$ is a reduction of $I$. 
For non-negative integer $n$, the numerical function of the length of $A/I^{n+1}$ is called the {\it Hilbert function} of $I$. For all $n\gg 0$, the Hilbert function forms a polynomial 
\[
\ell_A (A/I^{n+1})=\rme_0(I)\binom{n+d}{d}-\rme_1(I)\binom{n+d-1}{d-1}+ \cdots + (-1)^{d}\rme_d(I)
\]
in $n$ of degree $d$ and $\rme_0 (I), \rme_1 (I), \dots, \rme_d (I)$ are called the {\it Hilbert coefficients} of $I$.

With these assumptions and notations, Northcott \cite{No} shows that the inequality $\rme_1(I) \ge \rme_0(I) - \ell_A(A/I)$. After that, Huneke and Ooishi \cite{H, O} show that $\rme_1(I)=\rme_0 (I)- \ell_A (A/I)$ if and only if $I^2=QI$. When this is the case, the associated graded ring of $I$ is Cohen-Macaulay. As a continuous work of them, Sally \cite{S} investigated when the equality 
\begin{eqnarray}\label{rankone}
\rme_1(I)=\rme_0 (I)- \ell_A (A/I)+1
\end{eqnarray}
holds. Finally, Goto, Nishida, and Ozeki \cite{GNO2} characterized the equality (\ref{rankone}) by the form of the Sally module, which is introduced by Vasconcelos \cite{V}. 
In particular, they proved that $I^3=QI^2$ if the equality (\ref{rankone}) holds. On the other hand, the converse does not holds, and the difference $\rme_1(I)-\rme_0 (I) + \ell_A (A/I)$ can be an arbitrary non-negative integer even if $I^3=QI^2$ holds. 

In this paper, we study about the Hilbert function/coefficients of ideals with $I^3=QI^2$. Note that the reduction number depends on the choice of $Q$ in general, see for example \cite{Hu, Ma1, Ma2}. 
As known results on $I^3=QI^2$, the Hilbert function and the depth of the associated graded ring are classified if $I^3 = QI^2$ and one of the following conditions is satisfied:

\begin {itemize}
\item $\dim A=1$.
\item \cite{S, V} $\ell _A (I^2/QI) = 1$.
\item \cite[Theorem 1.2.]{GNO} $\fkm I^2 \subseteq QI$, $\ell _ A (I^2/QI) = 2$, and $\ell_A (I^3/Q^2 I)<2d$.
\item \cite[Theorem 3.6.]{CPR} $I$ is an integrally closed ideal.
\end {itemize}

Among them, in this paper, we focus on the conditions $I^3 = QI^2$ and $\fkm I^2 \subseteq QI$. In this case, we give the inequality
\[
\mathrm{e}_1(I)\ge \mathrm{e}_0(I) - \ell_A (A/I) + \mathrm{e}_2(I)
\] 
if $\dim A\ge 2$, see Theorem \ref{c3.3}. It is, of course, a stronger result than Northcott since $\rme_2(I)$ is non-negative by Narita \cite{Na}. Furthermore our inequality is, surprisingly, the converse inequality of Sally and Itoh \cite{S2, I}. We also give characterizations of the equalities 
\begin{center}
$\mathrm{e}_1(I)= \mathrm{e}_0(I) - \ell_A (A/I) + \mathrm{e}_2(I)\ \ $ and  $\ \ \mathrm{e}_1(I)= \mathrm{e}_0(I) - \ell_A (A/I) + \mathrm{e}_2(I)+1$. 
\end{center}
Besides them, we classify the Hilbert functions when $\ell _ A (I^2/QI)$ is either three or four.

Let us explain how this paper is organized. Section \ref{Preliminary} is a preparation to prove our main results. First, we see that the Bourbaki sequence for the case involving an infinite field. We then survey the basic properties of the Sally module, which is one of the important techniques to study the Hilbert functions. Section \ref{Main results} presents the main results of this paper. One can find some examples to illustrate our results.

\section{Preliminary}\label{Preliminary}

\subsection{Bourbaki sequence}\label{section2}

Let $R$ be a Noetherian ring and $M$ an $R$-module. 
Then a {\it Bourbaki sequence} of $M$ means a short exact sequence 
\[
0 \to F \to M \to I \to 0
\]
of $R$-modules, where $F$ is a free $R$-module and $I$ is an ideal of $R$. 
As a fundamental result, a Bourbaki sequence of $M$ always exists if $R$ is a normal domain and $M$ is a finitely generated torsionfree $R$-module (see \cite[Chapter VII, \S 4, 9. Theorem 6.]{B}). 
In this subsection, we investigate a Bourbaki sequence for the case involving an infinite field. An essential part of the result in this subsection is already proven in Bourbaki \cite[Chapter VII, \S 4, 9. Theorem 6.]{B}, but let us present a proof for the sake of completeness.

Throughout this subsection, let $R$ be a Noetherian ring and $M$ a finitely generated $R$-module. Suppose that $R$ contains an infinite field $k$ as a subring. 
For convenience, $M(\fkp)$ denotes $M_{\fkp}/\fkp M_{\fkp}$ for all prime ideals $\fkp$ of $R$. For an element $x\in M$, we denote by $x(\fkp)$ the canonical image of $x$ into $M(\fkp)$. Note that, for all prime ideal $\fkp$ of $R$, the canonical map $k\to R \to R(\fkp)$ is injective since $k\cap \fkp=(0)$. Hence we may assume $k\subseteq R(\fkp)$. We start with the following.


\begin{lem}\label{c2.1}
Let $\calP$ be a subset of the support $\Supp_R M=\{\fkp\in \Spec R \mid M_\fkp\ne0\}$ of $M$. Set $M=(x_1, x_2, \dots, x_n)$. Then the following assertions hold true. 
\begin{enumerate}[{\rm (1)}] 
\item Let $x$ and $y$ be elements of $M$. Set
\[
\calP'=\{ \fkp\in \calP \mid \text{$x(\fkp)$ and $y(\fkp)$ are linearly dependent over $R(\fkp)$} \}.
\]
Assume that $\calP'$ is finite and, for all $\fkp \in \calP'$, $x(\fkp)\ne 0$ or $y(\fkp)\ne0$. Then there is an element $a$ in $k$ such that $(x+ay)(\fkp)$ is nonzero for all $\fkp \in \calP$.
\item If $\calP$ is finite, then we can choose $y\in \sum_{i=1}^{n} kx_i$ so that $y(\fkp)$ is nonzero for all $\fkp\in \calP$. 
\item Suppose $R$ is a domain and $r=\rank_R M\ge 2$. Set 
\[
\calP=\{ \fkp\in \Supp_R M \mid \depth R_\fkp \le 1\}.
\]
Then we can choose $z\in \sum_{i=1}^{n} kx_i$ so that $z(\fkp)$ is nonzero for all $\fkp\in \calP$. 
\end{enumerate}
\end{lem}

\begin{proof}
(1) If $\fkp \in \calP\setminus \calP'$, then $(x+ ay)(\fkp)$ is nonzero for all $a\in k$. Hence we may assume that $\calP=\calP'$. We prove our assertion by induction on the number of element in $\calP$. If $\Card(\calP)=1$, it is trivial.
Assume  $\Card(\calP)>1$ and take $\fkp_0$ in $\calP$. Then there is an element $b$ in $k$ such that $(x+by)(\fkp)$ is nonzero for all $\fkp \in \calP\setminus \{ \fkp_0\}$ by induction hypothesis. Set $z=x+by$.
We may assume that $z(\fkp_0)=0$. Then $y(\fkp_0)$ is nonzero. Set 
\[
\calP''=\{ \fkp\in \calP \mid y(\fkp)\ne 0 \}.
\]
For all $\fkp \in \calP''$, there exists an element $\alpha_\fkp \in R(\fkp)$ such that $x(\fkp)=\alpha_\fkp{\cdot}y(\fkp)$ since $x(\fkp)$ and $y(\fkp)$ are linearly dependent. We can choose $c\in k$ so that $c\ne \alpha_{\mathfrak{p}}+b$ for all $\mathfrak{p}\in \mathcal{P}$ 
since $k$ is infinite. Then $a=b-c$ is what we desired. In fact, if $\fkp \in \calP''$, 
\[
(x+ay)(\fkp)=(\alpha_\fkp + b - c){\cdot}y(\fkp)\ne 0.
\]
If $\fkp\in \calP\setminus \calP''$, we have $y(\fkp)=0$ and $\fkp\ne \fkp_0$. Hence 
\[
(x+ay)(\fkp)=x(\fkp)=z(\fkp)\ne 0.
\]

(2) follows from (1).

(3) Let $K$ denote the quotient field of $R$. After renumbering $x_1, x_2, \dots, x_n$, we may assume that $1\otimes x_1$, $1\otimes x_2$, $\dots$, $1\otimes x_r$ is a $K$-free basis of $K\otimes_R M$. Set $L=\sum_{i=1}^{r} Rx_i$. Note that $L$ is a free $R$-module of rank $r$. We can choose a non-zerodivisor $a$ of $R$ so that $aM\subseteq L$ since $K\otimes_R (M/L)=0$. For a prime ideal $\fkp$ of $R$, if $a\not\in \fkp$, then $x_1(\fkp), x_2(\fkp), \dots, x_r(\fkp)$ are linearly independent over $R(\fkp)$ since $M_\fkp=L_\fkp$. On the other hand, by (2), we can choose $y\in \sum_{i=1}^{n} kx_i$ so that $y(\fkp)$ is nonzero for all $\fkp\in \{0\}\cup\Ass_R R/aR$. 
After renumbering $x_1, x_2, \dots, x_r$ if necessary, we may assume that $1\otimes x_1$, $\dots$, $1\otimes x_{r-1}$, $1\otimes y$ is a $K$-free basis of $K\otimes_R M$. Set $L'=\sum_{i=1}^{r-1} Rx_i + Ry$. Let $b$ be a non-zerodivisor of $R$ such that $bM\subseteq L'$. Then, for a prime ideal $\fkp$ with $b\not\in \fkp$,  $x_1(\fkp), x_2(\fkp), \dots, x_{r-1}(\fkp), y(\fkp)$ are linearly independent over $R(\fkp)$. Hence the set 
\[
\calP'=\{ \fkp\in \calP \mid \text{$x_1(\fkp)$ and $y(\fkp)$ are linearly dependent over $R(\fkp)$} \}
\]
is finite since $\calP' \subseteq \{ \fkp \in \calP \mid b\in \fkp\} \subseteq \Ass_R R/bR$. We furthermore have 
\[
\begin{cases}
x_1(\fkp)\ne0 & \text{if $a\not\in \fkp$}\\
y(\fkp)\ne0 & \text{if $a\in \fkp$}
\end{cases}
\]
for all $\fkp\in \calP'$. Therefore we have the conclusion by (1).
\end{proof}

\begin{lem}{\rm (cf. \cite[Chapter VII, \S 4, 9. Lemma 7.]{B})}\label{a2.2}
Let $R$ be a Noetherian domain and $M$ a finitely generated torsionfree $R$-module of rank $r\ge 2$. Suppose that $M_\fkp$ is $R_\fkp$-free for all prime ideals $\fkp$ with $\depth R_\fkp=1$.
Then, for $z\in M$, the following conditions are equivalent.
\begin{enumerate}[{\rm (1)}] 
\item $Rz$ is $R$-free and $M/Rz$ is a torsionfree $R$-module.
\item $z(\fkp)$ is nonzero for all $\fkp\in \Spec R$ with $\depth R_\fkp\le 1$.
\end{enumerate}
When this is the case, $\rank_R M/Rz=r-1$ and $(M/Rz)_{\mathfrak{p}}$ is $R_{\mathfrak{p}}$-free for all $\mathfrak{p}\in \Spec R$ with $\depth R_{\mathfrak{p}}\le 1$.
\end{lem}

\begin{proof}
(1) $\Rightarrow$ (2) For all prime ideal $\fkp$ with $\depth R_\fkp=1$, let
\begin{equation}\label{c2.2.1}
0\to (Rz)_\fkp \to M_\fkp \to (M/Rz)_\fkp \to 0
\end{equation}
be the exact sequence of $R_\fkp$-modules. The projective dimension of $(M/Rz)_{\mathfrak{p}}$ is at most one by (\ref{c2.2.1}) and our hypothesis. If $\depth_{R_\fkp} (M/Rz)_\fkp=0$, then $\fkp\in \Ass_R M/Rz \subseteq \{0\}$. It follows  a contradiction. Hence $(M/Rz)_\fkp$ is $R_\fkp$-free. It follows that the short exact sequence (\ref{c2.2.1}) splits and $z(\fkp)$ is nonzero. 
For the case where $\fkp=0$, $z(\fkp)$ is nonzero since $Rz\cong R$.

(2) $\Rightarrow$ (1) Let $K$ be the field of fractions of $R$. Since $z(0)$ is nonzero, $K\otimes_R Rz \cong K$. Hence $Rz\cong R$. 
Assume that $M/Rz$ is not a torsionfree $R$-module. Then there exists a prime ideal $\fkp \in \Ass_R (M/Rz) \setminus \{0\}$. By depth lemma, $\depth R_\fkp=1$ since $\depth_{R_\fkp} M_\fkp>0$. Therefore, since $z/1$ is a part of free basis of $M_\fkp$, the canonical map $(Rz)_\fkp \to M_\fkp$ is split mono. Thus $(M/Rz)_\fkp$ is $R_\fkp$-free, which is a contradiction for $\fkp\in \Ass_R (M/Rz)$ and $\depth R_\fkp=1$.
\end{proof}

Combining these two of lemmas, we have the following.

\begin{thm}\label{a2.1}
Let $R$ be a Noetherian domain and assume that $R$ contains an infinite field $k$ as a subring. Let $M=(x_1, x_2, \dots, x_n)$ be a torsionfree $R$-module of rank $r>0$. Suppose that $M_\fkp$ is $R_\fkp$-free for all $\fkp\in \Spec R$ with $\depth R_\fkp=1$. Then there are elements $z_1$, $\dots$, $z_{r-1}$ in $\sum_{i=1}^{n} kx_i$ which satisfy the following two conditions:
\begin{enumerate}[{\rm (1)}] 
\item $L= \sum_{i=1}^{r-1} Rz_i$ is a free $R$-module of rank $r-1$ and
\item $M/L$ is a torsionfree $R$-module of rank one.
\end{enumerate}
\end{thm}

\begin{proof}
We prove by induction on $r$. The case where $r=1$ is trivial. Assume that $r>1$ and our assertion holds for $r-1$. By Lemma \ref{c2.1} (3), we can choose $z_1\in \sum_{i=1}^{n} kx_i$ so that $z_1(\fkp)$ is nonzero for all $\fkp\in \Spec R$ with $\depth R_\fkp \le 1$. Hence $Rz_1$ is $R$-free and $M/Rz_1$ is a torsionfree $R$-module by Lemma \ref{a2.2}. By induction hypothesis, there exist $z_2$, $\dots$, $z_{r-1}$ in $\sum_{i=1}^{n} kx_i$ which satisfy the following two conditions:
\begin{enumerate}[{\rm (1)}] 
\item $\sum_{i=2}^{r-1} R\ol{z_i}$ is a free $R$-module of rank $r-2$, where $\ol{z_i}$ denotes the image of $z_i$ into $M/Rz_1$.
\item $M/L$ is a torsionfree $R$-module of rank one, where $L=\sum_{i=1}^{r-1} R z_i$.
\end{enumerate}
Then $z_1$, $\dots$, $z_{r-1}$ are what we desired since $L$ is an $R$-free module of rank $r-1$ by the split exact sequence
\[
0\to Rz_1 \to L \to \sum_{i=2}^{r-1} R\ol{z_i} \to 0
\]
of $R$-modules.
\end{proof}

As a direct consequence, we have a graded version of Bourbaki sequence.

\begin{cor}\label{c2.4}
Let $R=\bigoplus_{n\in \mathbb{Z}} R_n$ be a $\mathbb{Z}$-graded Noetherian domain and $M=\bigoplus_{n\in \mathbb{Z}} M_n$ a finitely generated graded $R$-module of rank $r>0$. Suppose that $R_0$ is an infinite field and $M=RM_1$. If $M$ is a torsionfree $R$-module and $M_\fkp$ is a free $R_\fkp$-module for all $\fkp\in \Spec R$ with $\depth R_\fkp=1$, then there exists a graded exact sequence
\[
0 \to R(-1)^{\oplus(r-1)} \to M \to I(m) \to 0
\]
of $R$-modules, where $m$ is an integer and $I$ is a graded ideal of $R$.
\end{cor}


\subsection{Sally module}\label{subsection2.2}

Let $(A, \fkm, k)$ be a local Cohen-Macaulay ring of dimension $d>0$. For simplicity, we assume that the residue field $k$ is infinite. Let $I$ be an $\fkm$-primary ideal of $A$ and $Q$ a minimal reduction of $I$. 
Let 
\begin{center}
$\calR=A[It]$ and $\calT=A[Qt]$
\end{center}
denote the {\it Rees algebras} of $I$ and $Q$ respectively, where $t$ stands for an indeterminate over $A$. We denote by $G=\calR/I\calR=\bigoplus_{n\ge 0} I^n/I^{n+1}$ the {\it associated graded ring} of $I$.
A finitely generated graded $\calT$-module 
\[
\calS_Q (I)=I\calR/I\calT=\bigoplus_{n\ge1} I^{n+1}/IQ^n
\]
is called the {\it Sally module} of $I$ with respect to $Q$. Set $\calS=\calS_Q (I)$.

The importance of the Sally module is a relationship with the Hilbert function and the depth of the associated graded ring $G$. 
Let $\rme_{i}=\rme_{i} (I)$ denote the $i$th Hilbert coefficient of $I$.
The following results are fundamental.

\begin{prop}{\rm (\cite[Lemma 2.1. and Proposition 2.2.]{GNO})}\label{a3.1}
The following assertions hold true.
\begin{enumerate}[{\rm (1)}] 
\item Set $\calM=\fkm \calT+\calT_+$. For $n>0$, 
\begin{center}
$(\calS/\calM \calS)_n=0$ if and only if $I^{n+1}=QI^{n}$.
\end{center}
\item $\Ass_\calT \calS\subseteq \{\fkm \calT \}$. Hence either $\calS=0$ or $\dim \calS=d$.
\item $\ell_A (A/I^{n+1})=\rme_0{\cdot}\binom{n+d}{d}-(\rme_0-\ell_A (A/I)){\cdot}\binom{n+d-1}{d-1}- \ell_A(\calS_n)$ for all $n\ge 0$.
\item $\rme_1=\rme_0-\ell_A (A/I)+\ell_{\calT_\fkp}(\calS_\fkp)$, where $\fkp=\fkm \calT$.
\item 
\begin{enumerate}[{\rm (i)}] 
\item If $\calS=0$, then $G$ is a Cohen-Macaulay ring.
\item If $\calS\ne0$ and $\depth_\calT \calS<d$, $\depth G=\depth_\calT \calS-1$. 
\item Assume $\calS\ne 0$. Then $\depth_\calT \calS=d$ if and only if $\depth G\ge d-1$.
\end{enumerate}
\end{enumerate}
\end{prop}

Note that Proposition \ref{a3.1} yields the result of Huneke and Ooishi, that is, $I^2=QI$ if and only if $\rme_1=\rme_0-\ell_A (A/I)$. It is equivalent to $\calS=0$. 
Like this, the structure of the Sally module determines the Hilbert function and the depth of $G$.
In the next section, we will classify the Hilbert function and the depth of $G$ through the structure of the Sally module.


\section{Main results}\label{Main results}

In this section, we maintain the assumptions and notations in Subsection \ref{subsection2.2}. Besides them, set 
\begin{center}
$B=\calT/\fkm \calT\cong k[X_1, X_2, \cdots, X_d]$ and $\fkp=\fkm \calT$,
\end{center}
where $X_1, X_2, \cdots, X_d$ denote indeterminates over the residue field $k$.

\begin{lem}\label{b3.1}
The following conditions are equivalent.
\begin{enumerate}[{\rm (1)}] 
\item $I^3=QI^2$ and $\fkm I^2\subseteq QI$.
\item $\calS=\calT \calS_1$ and $\fkm \calS=0$.
\end{enumerate}
When this is the case, $\calS$ is a torsionfree $B$-module of rank $\ell_{\calT_\fkp}(\calS_\fkp)$.
\end{lem}

\begin{proof}
The equivalence between $I^3=QI^2$ and $\calS=\calT \calS_1$ follows from Proposition \ref{a3.1} (1). 
The rest equivalence follows from the fact $\fkm I^2\subseteq QI$ if and only if $\fkm \calS_1=0$.
When this is the case, $\calS$ is a torsionfree $B$-module by Proposition \ref{a3.1}(2).
\end{proof}

Let $d=\dim A \ge 2$ and suppose that $I^2\ne QI$, $I^3=QI^2$, and $\fkm I^2\subseteq QI$. Set $\ell=\ell_{\calT_\fkp}(\calS_\fkp)>0$. By Corollary \ref{c2.4}, we then have a graded exact sequence 
\[
0 \to B(-1)^{\oplus(\ell-1)} \to \calS \to J(m) \to 0
\]
of $B$-modules, where $m$ is an integer and $J$ is a graded ideal of $B$.
We may assume that  $\height_B J \ge 2$ if $J\ne B$ since $B$ is a factorial domain. With these assumptions and notations, we have the following, which is the key of this section.

\begin{prop}\label{b3.2}
Let $d\ge 2$. Suppose that $I^2\ne QI$, $I^3=QI^2$, and $\fkm I^2\subseteq QI$. Let 
\[
0 \to B(-1)^{\oplus(\ell-1)} \to \calS \to J(m) \to 0 \eqno{(*)}
\]
be a graded exact sequence of $B$-modules such that $\height_B J\ge 2$.
Then 
\[
m=\ell_A(A/I)-\rme_0+\rme_1-\rme_2-1.
\]
In particular, $m$ is independent of the choice of $(*)$ and $-1\le m\le \ell-1$.
\end{prop}

\begin{proof}
Note that $J=BJ_{m+1}$ and $J_{m+1}\ne0$ by an exact sequence $(*)$ and Lemma \ref{b3.1}. Hence $m=-1$ if $J=B$. Then $\calS\cong B(-1)^{\oplus \ell}$, whence $\rme_2=\ell=\ell_A(A/I)-\rme_0+\rme_1$ by Proposition \ref{a3.1} (3) and (4).

Suppose that $J\subsetneq B$. Then $m+1\ge 1$. Consider the exact sequence
\[
0 \to B(-1)^{\oplus(\ell-1)} \to \calS \to B(m) \to (B/J)(m) \to 0
\]
of graded $B$-modules. Thus
\[
\ell_A (\calS_n)+\ell_A((B/J)_{m+n})=(\ell-1){\cdot}\ell_A(B_{n-1})+\ell_A(B_{n+m})
\] 
for all $n\in \mathbb{Z}$. Note that the degree of $\ell_A((B/I)_{m+n})$ is at most $d-3$ since the height of $J$ is at least two. Therefore
{\small
\begin{align*}
\ell_A (\calS_n) &=(\ell-1)\left\{  \binom{n+d-1}{d-1} - \binom{n+d-2}{d-2}\right\} + \binom{n+d-1}{d-1}+m\binom{n+d-2}{d-2} + (\text{lower term})\\
 &= \ell{\cdot} \binom{n+d-1}{d-1}- (\ell-1-m){\cdot}\binom{n+d-2}{d-2}+ (\text{lower term}).
\end{align*}
}
Hence, by Proposition \ref{a3.1}(3) and (4), we have the equality 
\[
\rme_2=\ell-1-m=\rme_1-\rme_0+\ell_A (A/I)-1 -m,
\]
which is an non-negative integer by Narita's theorem \cite{Na}.
\end{proof}

The following is a direct consequence of Proposition \ref{b3.2}.

\begin{thm}\label{c3.3}
Suppose that $d\ge 2$.
If $I^3=QI^2$ and $\fkm I^2\subseteq QI$, then we have the inequality
\[
\mathrm{e}_1\ge \mathrm{e}_0- \ell_A (A/I) + \mathrm{e}_2
\]
\end{thm}

\begin{proof}
If $I^2=QI$, then $\mathrm{e}_1= \mathrm{e}_0- \ell_A (A/I)$ and $\mathrm{e}_2=0$ by \cite{H, O}. Hence we may assume that $I^2\ne QI$. In this case, our assertion follows from the inequality $-1\le m =\ell_A(A/I)-\rme_0+\rme_1-\rme_2-1$ by Proposition \ref{b3.2}.
\end{proof}

The bounds in Proposition \ref{b3.2} and Theorem \ref{c3.3} are sharp. Let us note two examples.

\begin{ex}
Let $A=k[[X, Y]]$ be a formal power series ring over a field $k$ and $\ell$ a positive integer. Set 
\begin{center}
$Q=(X^{2\ell+2}, Y^{2\ell+2})$ and $I=Q + (XY^{2\ell+1}, X^3Y^{2\ell-1}, \dots , X^{2i+1}Y^{2\ell-2i+1}, \dots , X^{2\ell+1}Y)$.
\end{center}
Then $I^3=QI^2$, $(X, Y)I^2\subseteq QI$, and 
$\ell_A (A/I^{n+1})=4(\ell+1)^2\binom{n+2}{2} -(2\ell^2+3\ell + 1)\binom{n+1}{1}$ for all $n\ge 1$.
Thus $\rme_1=\rme_0 - \ell_A (A/I) + \ell$ and $\rme_2=0$. We furthermore have the graded exact sequence 
\[
0\to B(-1)^{\oplus (\ell-1)}\to \calS_Q(I) \to (X, Y)^{\ell}B (\ell-1) \to 0
\]
as $B=k[X, Y]$-modules and $\depth G(I)=0$.
\end{ex}

\begin{proof}
It is routine to show that $I^3=QI^2$, $(X, Y)I^2\subseteq QI$.
It is also easy to show that $I^n=\ol{I^n}=(X, Y)^{2(\ell+1)n}$ for all $n\ge 2$, where $\ol{I^n}$ denotes the integral closure of $I^n$. Hence 
\begin{align*}
\ell_A (A/I^{n+1})&=n(\ell+1)\{ 2n(\ell + 1) + 1\}\\
&=4(\ell+1)^2\binom{n+2}{2} -(2\ell^2+3\ell + 1)\binom{n+1}{1}
\end{align*}
for all $n\ge 1$. The rank of the Sally module $\calS=\calS_Q(I)$ is 
\[
\rme_1 (I) - \rme_0 (I) + \ell_A (A/I)=(2\ell^2+3\ell + 1) - 4(\ell+1)^2 + (2\ell^2 + 4\ell +3)=\ell.
\]
Thus we have the graded exact sequence 
\[
0\to B(-1)^{\oplus (\ell-1)}\to \calS \to J(m) \to 0
\]
of $B$-modules for some graded ideal $J$ and some integer $-1\le m\le \ell-1$. We furthermore have 
$\ell_A (J_{m+1})=\ell_A (\calS_1)-(\ell-1)=2\ell-(\ell-1)=\ell+1$. Thus $m=\ell-1$ and $J=(X, Y)^{\ell}$ since $\ell_A (J_{m+1})\le \ell_A (B_{m+1}) \le \ell_A (B_{\ell})=\ell+1$.
Hence the depth of the associated graded ring is zero by Proposition \ref{a3.1}(5)(ii).
\end{proof}

\begin{ex}
Let $A=k[[X, Y]]$ be a formal power series ring over a field $k$. 
Set $Q=(X^5, Y^5)$ and $I=Q + (X^2Y^3, X^3Y^2)$. Then ${I}^3=Q{I}^2$, $\fkm {I}^2\subseteq QI$, and $\ell_A ({I}^{n+1}/Q^{n}I)=2n$ for all $n\ge 1$. Hence, $\calS_{Q} (I) \cong k[X, Y](-1)^{\oplus 2}$.

\end{ex}

\begin{proof}
Since ${I}^3=Q{I}^2$ and $\fkm {I}^2\subseteq QI$, there is a surjection $k[X, Y](-1)^{\oplus 2} \to \calS_{Q} (I)$. Therefore we have the isomorphism since the kernel is zero by $\ell_A ({I}^{n+1}/Q^{n}I)=2n$ for all $n\ge 1$.
\end{proof}

Examples that $-1<m<\ell_A(A/I)-\rme_0+\rme_1-\rme_2-1$ also exist. We will see it later (Example \ref{b3.8}). 
Here, let us give some applications of Proposition \ref{b3.2}. In what follows, we always assume that $\dim A \ge 2$.

\begin{cor}\label{b3.4}
Suppose that $I^3=QI^2$ and $\fkm I^2\subseteq QI$. Let $\ell$ denote the rank of the Sally module as $B$-module.
Then we have the following.
\begin{enumerate}[{\rm (1)}] 
\item The following conditions are equivalent.
\begin{enumerate}[{\rm (i)}] 
\item $\mathrm{e}_1= \mathrm{e}_0 - \ell_A (A/I) + \mathrm{e}_2.$
\item $\calS\cong B(-1)^{\oplus \ell}$.
\item $\depth G\ge d-1$.
\end{enumerate}
When this is the case, 
\begin{center}
$\ell_A (A/I^{n+1})=\rme_0{\cdot}\binom{n+d}{d} -\rme_1{\cdot}\binom{n+d-1}{d-1}+\ell{\cdot}\binom{n+d-2}{d-2}$ for all $n\ge 0$.
\end{center}
\item $\mathrm{e}_1= \mathrm{e}_0- \ell_A (A/I) + \mathrm{e}_2 + 1$ if and only if 
\[
0\to B(-1)^{\oplus (\ell-1)}\to \calS \to (X_1, X_2, \dots, X_c)B \to 0
\]
is exact as graded $B$-modules, where $2\le c=\ell_A(I^2/QI)-\ell+1\le d$.
When this is the case, $\depth G=d-c$ and 

\begin{flushleft}
{\small
\[\ell_A (A/I^{n+1})=\begin{cases}
\rme_0{\cdot}\binom{n+d}{d} -\rme_1{\cdot}\binom{n+d-1}{d-1}+(\ell-1){\cdot}\binom{n+d-2}{d-2}+ \binom{n+d-c-1}{d-c-1} \text{ for all $n\ge 0$} & \text{if $c<d$}.\\
\rme_0{\cdot}\binom{n+d}{d} -\rme_1{\cdot}\binom{n+d-1}{d-1}+(\ell-1){\cdot}\binom{n+d-2}{d-2} \text{ for all $n\ge 1$} & \text{if $c=d$}.
\end{cases}\]
}
\end{flushleft}
\end{enumerate}
\end{cor}

\begin{proof}
(1) Note that $\calS$ is a free $B$-module if and only if $\depth G\ge d-1$ by Proposition \ref{a3.1}(5)(iii). 
We have only to show the equivalence of (i) and (ii).
If $\rme_2 = \ell_A(A/I)-\rme_0+\rme_1=\ell$, then $m=-1$ and $I=B$ by Proposition \ref{b3.2}. Hence the exact sequence $(*)$ splits, thus $\calS\cong B(-1)^{\oplus \ell}$. If $\calS\cong B(-1)^{\oplus \ell}$, then $\ell_A (\calS_n)=\ell{\cdot} \left\{  \binom{n+d-1}{d-1} - \binom{n+d-2}{d-2}\right\}$ for all $n\ge 0$, whence we have the assertions by Proposition \ref{a3.1}(3).

(2) If $\rme_2 = \ell_A(A/I)-\rme_0+\rme_1-1$, then $m=0$ by Proposition \ref{b3.2}. Hence $J=\calT J_1\cong (X_1, X_2, \dots, X_c)B$, where $c=\ell_A(I^2/QI)-\ell+1$. If $c=1$, then $\calS\cong B(-1)^{\oplus \ell}$ and $\rme_2= \ell_A(A/I)-\rme_0+\rme_1$ by (1), which is a contradiction. Thus $c\ge 2$. Notice that $\calS$ is not a Cohen-Macaulay $\calT$-module since $\calS$ is not $B$-free, whence $\depth_\calT \calS=d-c+1$ by depth lemma. 
The rest assertions follow from Proposition \ref{a3.1}. The converse is now clear.
\end{proof}

\begin{cor}{\rm (cf. \cite[Theorem 3.6]{CPR})}\label{d3.7}
Suppose that $I^3=QI^2$ and $\fkm I^2\subseteq QI$. 
Assume that either 
\begin{enumerate}[{\rm (1)}] 
\item $I$ is an integrally closed ideal or
\item $d=2$ and $I=\tilde{I}$, where $\tilde{I}$ denotes the Ratliff-Rush closure of $I$ {\rm (\cite[Chapter VIII, Notation]{Mc})}.
\end{enumerate}

Then $\calS\cong B(-1)^{\oplus \ell}$, hence $\depth G\ge d-1$ and 
\begin{center}
$\ell_A (A/I^{n+1})=\rme_0{\cdot}\binom{n+d}{d} -\rme_1{\cdot}\binom{n+d-1}{d-1}+\ell{\cdot}\binom{n+d-2}{d-2}$ for all $n\ge 0$.
\end{center}
\end{cor}

\begin{proof}
This is a direct consequence of Corollary \ref{b3.4}(1) since $\rme_2\ge \ell_A(A/I)-\rme_0+\rme_1$ by \cite{I, S2}.
\end{proof}

Note that Corso, Polini, and Rossi \cite[Theorem 3.6]{CPR} do not assume that $\fkm I^2\subseteq QI$, and we have no example which does not hold the assertion of Theorem \ref{c3.3} without assumption $\fkm I^2\subseteq QI$. 

Next we determine the Hilbert function from ideal conditions.
The following is a continuation of Sally and Goto-Nishida-Ozeki (\cite{S} and \cite{GNO}).


\begin{thm}\label{b3.6}
The following conditions are equivalent.
\begin{enumerate}[{\rm (1)}] 
\item $I^3=QI^2$, $\fkm I^2\subseteq QI$, $\ell_A (I^2/QI)=3$, and $\ell_A(I^3/Q^2I)<3d$.
\item The graded minimal $B$-free resolution of the Sally module is either
\begin{enumerate}[{\rm (i)}] 
\item $0\to B(-2)\to B(-1)^{\oplus 3}\to \calS\to 0$ or
\item $0\to B(-3) \to B(-2)^{\oplus 3}\to B(-1)^{\oplus 3}\to \calS\to 0$,

i.e. $\calS\cong (X_1, X_2, X_3)B$.
\end{enumerate}
Here, the case {\rm (ii)} only occurs when $d\ge 3$.
\end{enumerate}
When this is the case, $\mathrm{e}_1= \mathrm{e}_0- \ell_A (A/I) + \mathrm{e}_2 + 1$ and we have the following.

\begin{enumerate}[{\rm (i)}] 
\item $\ell_A(I^3/Q^2I)=3d-1$, $\depth G=d-2$, $\rme_1=\rme_0-\ell_A (A/I)+2$, and
\begin{center}
$\ell_A (A/I^{n+1})=\rme_0{\cdot}\binom{n+d}{d} -\rme_1{\cdot}\binom{n+d-1}{d-1}+\binom{n+d-2}{d-2} - (-1){\cdot}\binom{n+d-3}{d-3}$ for all $n\ge 0$.
\end{center}
\item $\ell_A(I^3/Q^2I)=3d-3$, $\depth G=d-3$, $\rme_1=\rme_0-\ell_A (A/I)+1$, and
\begin{center}
$\ell_A (A/I^{n+1})=\rme_0{\cdot}\binom{n+d}{d} -\rme_1{\cdot}\binom{n+d-1}{d-1}+\binom{n+d-4}{d-4}$ for all $n\ge 0$ if $d\ge 4$.
\end{center}
\end{enumerate}


\end{thm}

\begin{proof}
We have only to show the implication $(1)\Rightarrow (2)$. Set $\ell=\ell_{\calT_\fkp}(\calS_\fkp)$. 
Since there is a surjection $B(-1)^{\oplus 3}\to \calS$, $\ell\le 3$. If $\ell=3$, then the surjection is isomorphism, which is a contradiction for $\ell_A (\calS_2)=\ell_A(I^3/Q^2I)<3d$. If $\ell=1$, $\calS\cong (X_1, X_2, X_3)B$ by \cite[Theorem 1.2.]{GNO2}. Hence we may assume $\ell=2$. By Proposition \ref{b3.2}, there exists $0\to B(-1) \to \calS \to J(m) \to 0$, where $\height_B J\ge 2$ and $m=0$ or $1$. Since $2\le \height_B J\le \mu_B(J)=2$, $J$ is generated by a regular sequence of degree $m+1$. Hence we have the graded exact sequence
$0\to B(-2m-2)\to B(-m-1)^{\oplus 2}\to J \to 0$ as $B$-modules, thus
\[
\xymatrix{
&&&0\ar[d]&\\
&&&B(-m-2)\ar[d]&\\
&&&B(-1)^{\oplus 2}\ar[d]&\\
0\ar[r] & B(-1) \ar[r] & \calS \ar[r] & J(m) \ar[r] \ar[d]& 0\\
&&&0 &.\\
}
\]
The pullback of the above diagram gives the short exact sequence in (i) since $m=0$ by $\ell_A(I^3/Q^2I)<3d$.
\end{proof}

Note that the rank of Sally module in Theorem \ref{b3.6}(2)(i) is two. In general, the structure of the Sally module of rank two is quite open (\cite[p.883]{GNO2} and \cite[4.4]{RV}). 
The case of $\ell_A (I^2/QI)=4$ can also be classified as the following cases if $\dim A=2$.

\begin{prop}\label{b3.65}
Suppose that $d=2$. Then the following conditions are equivalent.
\begin{enumerate}[{\rm (1)}] 
\item $I^3=QI^2$, $\fkm I^2\subseteq QI$, $\ell_A (I^2/QI)=4$, and $\ell_A(I^3/Q^2I)<8$.
\item The graded minimal $B$-free resolution of the Sally module is either
\begin{enumerate}[{\rm (i)}] 
\item $0\to B(-2)\to B(-1)^{\oplus 4}\to \calS\to 0$ or
\item $0\to B(-2)^{\oplus 2}\to B(-1)^{\oplus 4}\to \calS\to 0$.
%
\end{enumerate}
\end{enumerate}
When this is the case, $\depth G=0$ and we have the following.

\begin{enumerate}[{\rm (i)}] 
\item $\ell_A(I^3/Q^2I)=7$, $\rme_1=\rme_0-\ell_A (A/I)+3$, $\mathrm{e}_1= \mathrm{e}_0- \ell_A (A/I) + \mathrm{e}_2 + 1$, and
\begin{center}
$\ell_A (A/I^{n+1})=\rme_0{\cdot}\binom{n+2}{2} -\rme_1{\cdot}\binom{n+1}{1}+2$ for all $n\ge 1$.
\end{center}
\item $\ell_A(I^3/Q^2I)=6$, $\rme_1=\rme_0-\ell_A (A/I)+2$, $\mathrm{e}_1= \mathrm{e}_0- \ell_A (A/I) + \mathrm{e}_2 + 2$, and
\begin{center}
$\ell_A (A/I^{n+1})=\rme_0{\cdot}\binom{n+2}{2} - \rme_1{\cdot}\binom{n+1}{1}$ for all $n\ge 1$.
\end{center}
\end{enumerate}



\end{prop}


\begin{proof}
We have only to show the implication $(1)\Rightarrow (2)$. Set $\ell=\ell_{\calT_\fkp}(\calS_\fkp)$. 
Similar to the proof of Theorem \ref{b3.6}, we may assume that $\ell=2$ or $3$. If $\ell=3$, there exists an exact sequence $0\to B(-1)^{\oplus 2} \to \calS \to J(m) \to 0$, where $J$ is a graded ideal of $B$ and $0\le m\le 2$. We can take $J$ so that $\height_B J=\mu_B(J)= 2$, that is, $J$ is generated by a regular sequence of degree $m+1$.
Hence we have the graded exact sequence
\[
0\to B(-2m-2)\to B(-m-1)^{\oplus 2}\to J \to 0
\]
as $B$-modules, and this concludes $\calS$ has the structure in (i) since $\ell_A(I^3/Q^2I)<8$. Suppose that $\ell=2$. 
Then there exists an exact sequence $0\to B(-1) \to \calS \to J(m) \to 0$, where $J$ is a graded ideal with $\height_B J=2$, $\mu_B(J)=3$, and $0\le m\le 1$. Thus $m=1$ since $\ell_A (J_1)\le \ell_A (B_1) = 2$ and $\ell_A (I^2/QI)=4$. Hence we have the graded exact sequence
\[
0\to B(-n_1)\oplus B(-n_2)\xrightarrow{\mathbb{A}} B(-2)^{\oplus 3}\to I\to 0 \eqno{(**)}
\]
as graded $B$-modules, where $n_1$, $n_2$ are integers and $\mathbb{A}=\left(
\begin{smallmatrix}
a & d\\
b & e\\
c & f
\end{smallmatrix}
\right)$ denotes a representation matrix. Therefore, the graded minimal $B$-free resolution of the Sally module is
\[
0\to B(1-n_1)\oplus B(1-n_2)\to B(-1)^{\oplus 4}\to \calS\to 0.
\]
We will show that $n_1=n_2=3$.
Since $\ell_A (I^2/QI)=4$ and $\ell_A(I^3/Q^2I)<8$, we may assume that $n_1=3$. By $(**)$, $a$, $b$, $c\in B_1$ and $d$, $e$, $f\in B_{n_2-2}$ since 

$\left(
\begin{smallmatrix}
a \\
b \\
c
\end{smallmatrix}
\right)\in [ B(-2)^{\oplus 3}]_3$ and
$\left(
\begin{smallmatrix}
d \\
e \\
f
\end{smallmatrix}
\right)\in [ B(-2)^{\oplus 3}]_{n_2}$. On the other hand, $J$ is generated by $2\times 2$-minors of the matrix $\mathbb{A}$ by Hilbert-Burch theorem. By noting that $J=BJ_2$, we have $n_2=3$.
\end{proof}

\begin{rem}\label{b3.9}
If $d\ge 3$, there is a Sally module which have another form. In fact, 
let $A=k[[X, Y, Z]]$ be a formal power series ring over a field $k$. Set $Q=(X^3, Y^3, Z^3)$ and $I=Q + (X^2Y, XY^2, Y^2Z, YZ^2, X^2Z, XZ^2)$. Then $I^3=QI^2$, $\fkm I^2\subseteq QI$, $\ell_A (I^2/QI)=4$, and $\ell_A (I^3/Q^2I)=9$. Hence this example does not satisfy either of the conditions in Theorem \ref{b3.65}.

\end{rem}

We close this paper with examples of Theorem \ref{b3.6} and Proposition \ref{b3.65}. Note that the examples of rank one are given by \cite[Theorem 5.1]{GNO2}.

\begin{ex}\label{b3.8}
Let $A=k[[X, Y]]$ be a formal power series ring over a field $k$ and $B=k[X, Y]$ a polynomial ring over the field $k$. Then we have the following.
\begin{enumerate}[{\rm (1)}] 
\item Set $Q=(X^7, Y^7)$ and $I=Q + (XY^6, X^2Y^5, X^4Y^3, X^5Y^2)$. Then ${I}^3=Q{I}^2$, $\fkm {I}^2\subseteq~QI$, $\ell_A ({I}^2/QI)=3$, and $\ell_A ({I}^3/{Q}^2I)=5$. Hence, 
\[
0\to B(-2)\to B(-1)^{\oplus 3}\to \calS_{Q} (I)\to 0
\]
is exact as graded $B$-modules.
\item Set $Q=(X^8, Y^8)$ and $I=Q + (X^2Y^6, X^3Y^5, X^5Y^3, X^6Y^2)$. Then ${I}^3=Q{I}^2$, $\fkm {I}^2\subseteq~QI$, $\ell_A ({I}^2/QI)=4$, and $\ell_A ({I}^3/{Q}^2I)=7$. Hence, 
\[
0\to B(-2)\to B(-1)^{\oplus 4}\to \calS_{Q} (I)\to 0
\]
is exact as graded $B$-modules.
\item Set $Q=(X^7, Y^7)$ and $I=Q + (XY^6, X^3Y^4, X^4Y^3, X^6Y)$. Then ${I}^3=Q{I}^2$, $\fkm {I}^2\subseteq~QI$, $\ell_A ({I}^2/QI)=4$, and $\ell_A ({I}^3/{Q}^2I)=6$. Hence, 
\[
0\to B(-2)^{\oplus 2}\to B(-1)^{\oplus 4}\to \calS_{Q} (I)\to 0
\]
is exact as graded $B$-modules.
\end{enumerate}
\end{ex}





\begin{acknowledgments}
The author is grateful to Koji Nishida for suggestions about the manuscript, which he gave the author while preparing for the final version. The author would like to thank the referee for his/her valuable comments.
\end{acknowledgments}



\end{document}